\documentclass[11pt]{article}
\usepackage{amsmath,amsfonts,amsthm,hyperref}
\usepackage[left=1in,top=1in,right=1in]{geometry}
\usepackage[hyphenbreaks]{breakurl}

\theoremstyle{definition}
\newtheorem{theorem}{Theorem}[section]
\newtheorem{lemma}[theorem]{Lemma}
\newtheorem{claim}[theorem]{Claim}
\newtheorem{problem}[theorem]{Problem}
\counterwithin{equation}{section}
\allowdisplaybreaks

\begin{document}

\title{On higher dimensional point sets in general position}
\author{Andrew Suk\thanks{Department of Mathematics, University of California at San Diego, La Jolla, CA, 92093 USA. Supported by NSF CAREER award DMS-1800746 and NSF award DMS-1952786. Email: {\tt asuk@ucsd.edu}.}\and Ji Zeng\thanks{Department of Mathematics, University of California at San Diego, La Jolla, CA, 92093 USA. Supported by NSF grant DMS-1800746. Email:{\tt jzeng@ucsd.edu}.}}
\date{}

\maketitle

\begin{abstract}
A finite point set in $\mathbb{R}^d$ is in general position if no $d + 1$ points lie on a common hyperplane. Let $\alpha_d(N)$ be the largest integer such that any set of $N$ points in $\mathbb{R}^d$, with no $d + 2$ members on a common hyperplane, contains a subset of size $\alpha_d(N)$ in general position. Using the method of hypergraph containers, Balogh and Solymosi showed that $\alpha_2(N) < N^{5/6 + o(1)}$.  In this paper, we also use the container method to obtain new upper bounds for $\alpha_d(N)$ when $d \geq 3$. More precisely, we show that if $d$ is odd, then $\alpha_d(N) < N^{\frac{1}{2} + \frac{1}{2d} + o(1)}$, and if $d$ is even, we have $\alpha_d(N) < N^{\frac{1}{2} + \frac{1}{d-1} + o(1)}$. We also study the classical problem of determining $a(d,k,n)$, the maximum number of points selected from the grid $[n]^d$ such that no $k + 2$ members lie on a $k$-flat, and improve the previously best known bound for $a(d,k,n)$, due to Lefmann in 2008, by a polynomial factor when $k$ = 2 or 3 (mod 4).
\end{abstract}

\section{Introduction}
A finite point set in $\mathbb{R}^d$ is said to be in \emph{general position} if no $d + 1$ members lie on a common hyperplane.  Let $\alpha_d(N)$ be the largest integer such that any set of $N$ points in $\mathbb{R}^d$, with no $d + 2$ members on a hyperplane, contains $\alpha_d(N)$ points in general position.  

In 1986, Erd\H os \cite{erdos} proposed the problem of determining $\alpha_2(N)$ and observed that a simple greedy algorithm shows $\alpha_2(N) \geq \Omega(\sqrt{N})$.  A few years later, F\"uredi \cite{F} showed that\begin{equation*}
    \Omega(\sqrt{N\log N}) < \alpha_2(N) < o(N),
\end{equation*} where the lower bound uses a result of Phelps and R\"odl \cite{PR} on partial Steiner systems, and the upper bound relies on the density Hales-Jewett theorem \cite{F1,F2}. In 2018, a breakthrough was made by Balogh and Solymosi \cite{BS}, who showed that $\alpha_2(N) < N^{5/6+o(1)}$.  Their proof was based on the method of hypergraph containers, a powerful technique introduced independently by Balogh, Morris, and Samotij \cite{BMS} and by Saxton and Thomason \cite{ST}, that reveals an underlying structure of the independent sets in a hypergraph. We refer interested readers to \cite{BMSs} for a survey of results based on this method.

In higher dimensions, the best lower bound for $\alpha_d(N)$ is due to Cardinal, T\'oth, and Wood \cite{CTW}, who showed that $\alpha_d(N) \geq \Omega((N\log N)^{1/d})$, for every fixed $d\geq 2$. For upper bounds, Mili{\'c}evi{\'c} \cite{M} used the density Hales-Jewett theorem to show that $\alpha_d(N) = o(N)$ for every fixed $d\geq 2$. However, these upper bounds in \cite{M}, just like those in \cite{F}, are still almost linear in $N$. Our main result is the following.
\begin{theorem}\label{main}
Let $d\geq 3$ be a fixed integer. If $d$ is odd, then $\alpha_d(N) < N^{\frac{1}{2} + \frac{1}{2d} + o(1)}$. If $d$ is even, then $\alpha_d(N) < N^{\frac{1}{2} + \frac{1}{d-1} + o(1)}.$
\end{theorem}

\noindent  Our proof of Theorem \ref{main} is also based on the hypergraph container method. A key ingredient in the proof is a new supersaturation lemma for $(k + 2)$-tuples of the grid $[n]^d$ that lie on a $k$-flat, which we shall discuss in the next section. Here, by a \emph{$k$-flat} we mean a $k$-dimensional affine subspace of $\mathbb{R}^d$.

One can consider a generalization of the quantity $\alpha_d(N)$. We let $\alpha_{d,s}(N)$ be the largest integer such that any set of $N$ points in $\mathbb{R}^d$, with no $d + s$ members on a hyperplane, contains $\alpha_{d,s}(N)$ points in general position. Hence, $\alpha_d(N) = \alpha_{d,2}(N)$.  A simple argument of Erd\H os \cite{erdos} shows that $\alpha_{d,s}(N) \geq \Omega(N^{1/d})$ for fixed $d$ and $s$ (see Section~\ref{sec_remarks}, or \cite{CTW} for large $s$). 
 In the other direction, following the arguments in our proof of Theorem~\ref{main} with a slight modification, we show the following.
\begin{theorem}\label{thm_remark}
Let $d,s\geq 3$ be fixed integers. If $d$ is odd and $ds + 2 > 2d + 2s$, then $\alpha_{d,s}(N)\leq N^{\frac{1}{2}+o(1)}$. If $d$ is even and $ds + 2 > 2d + 3s$, then $\alpha_{d,s}(N)\leq N^{\frac{1}{2}+o(1)}$.
\end{theorem}
\noindent For example, when we fix $d=3$ and $s\geq 5$, we have $\alpha_{d,s}(N)\leq N^{\frac{1}{2}+o(1)}$.

\medskip

We also study the classical problem of determining the maximum number of points selected from the grid $[n]^d$ such that no $k + 2$ members lie on a $k$-flat. The key ingredient of Theorem~\ref{main} mentioned above can be seen as a supersaturation version of this Tur{\'a}n-type problem. When $k=1$, this is the famous \emph{no-three-in-line problem} raised by Dudeney \cite{Du} in 1917: Is it true that one can select $2n$ points in $[n]^2$ such that no three are collinear? Clearly, $2n$ is an upper bound as any vertical line must contain at most 2 points. For small values of $n$, many authors have published solutions to this problem obtaining the bound of $2n$ (e.g. see \cite{Fl}), but for large $n$, the best known general construction is due to Hall--Jackson--Sudbery--Wild~\cite{Ha} with slightly fewer than $3n/2$ points.

More generally, we let $a(d,k,r,n)$ denote the maximum number of points from $[n]^d$ such that no $r$ points lie on a $k$-flat. Since $[n]^d$ can be covered by $n^{d-k}$ many $k$-flats, we have the trivial upper bound $a(d,k,r,n) \leq (r-1)n^{d-k}$. For certain values $d$, $k$, and $r$ fixed and $n$ tends to infinity, this bound is known to be asymptotically best possible: Many authors \cite{Roth,BK,L} noticed that $a(d,d-1,d+1,n) = \Theta(n)$ by looking at the modular moment curve over a finite field $\mathbb{Z}_p$; In \cite{PW}, P\'or and Wood proved that $a(3,1,3,n)=\Theta(n^2)$; Dvir and Lovett~\cite{dvir2012subspace} showed that $a(d,k,r,n) = \Theta(n^{d-k})$ when $r > d^k$ (see also \cite{TS}).

We shall focus on the case when $r = k + 2$ and write $a(d,k,n):=a(d,k,k+2,n)$. Surprisingly, Lefmann \cite{L} (see also \cite{L08}) showed that $a(d,k,n)$ behaves much differently than $\Theta(n^{d-k})$. In particular, he showed that\begin{equation*}
    a(d,k,n) \leq O\left(n^{\frac{d}{\lfloor (k + 2)/2\rfloor}}\right).
\end{equation*} Our next result improves this upper bound when $k+2$ is congruent to 0 or 1 mod 4.
\begin{theorem}\label{main2}
For fixed $d$ and $k$, as $n\to\infty$, we have\begin{equation*}
    a(d,k,n)\leq O\left(n^{\frac{d}{2\lfloor (k+2)/4\rfloor}(1-\frac{1}{2\lfloor(k+2)/4\rfloor d+1})}\right).
\end{equation*}
\end{theorem}
\noindent For example, we have $a(4,2,n)\leq O(n^{\frac{16}{9}})$ while Lefmann's bound in \cite{L} gives us $a(4,2,n)\leq O(n^{2})$, which coincides with the trivial upper bound. In particular, Theorem~\ref{main2} tells us that, if $4$ divides $k+2$, then $a(d,k,n)$ only behaves like $\Theta(n^{d-k})$ when $d=k+1$. This is quite interesting compared to the fact that $a(3,1,n)=\Theta(n^2)$ proved in \cite{PW}. Lastly, let us note that the current best lower bound for $a(d,k,n)$ is also due to Lefmann \cite{L}, who showed that $a(d,k,n) \geq \Omega\left(n^{\frac{d}{k + 1} - k - \frac{k}{k + 1}}\right)$.
 
For integer $n > 0$, we let $[n] = \{1,\dots, n\}$, and $\mathbb{Z}_n = \{0,1,\dots, n-1\}$. We systemically omit floors and ceilings whenever they are not crucial for the sake of clarity in our presentation. All exponentials and logarithms are in base two.

\section{Supersaturation of non-degenerate coplanar tuples}

In this section, we establish some lemmas for the proofs of Theorems~\ref{main} and~\ref{thm_remark}.

Given a set $T$ of $k + 2$ points in $\mathbb{R}^d$ that lie on a $k$-flat, we say that $T$ is \emph{degenerate} if there is a subset $S\subset T$ of size $j$, where $3 \leq j \leq k + 1$, such that $S$ lies on a $(j-2)$-flat. Otherwise, we say that $T$ is \emph{non-degenerate}. We establish a supersaturation lemma for non-degenerate $(k + 2)$-tuples of $[n]^d$.
\begin{lemma}\label{supersaturation}
For real number $\delta > 0$ and fixed positive integers $d,k$, such that $k$ is even and $d - 2\delta > (k -1)(k + 2)$, any subset $V\subset [n]^d$ of size $n^{d-\delta}$ spans at least $\Omega(n^{(k + 1)d - (k + 2)\delta})$ non-degenerate $(k+2)$-tuples that lie on a $k$-flat.
\end{lemma}
\begin{proof}
Let $V\subset [n]^d$ such that $|V| = n^{d - \delta}$. Set $r  = \frac{k}{2} + 1$ and $E_r  = \binom{V}{r}$ to be the collection of $r$-tuples of $V$. Notice that the sum of an $r$-tuple from $V$ belongs to $[rn]^d$. For each $v \in [rn]^d$, we define\begin{equation*}
    E_r(v)=\{\{v_1,\dots,v_r\}\in E_r: v_1+\dots+v_r=v\}.
\end{equation*} Then for $T_1,T_2 \in E_r(v)$, where $T_1 = \{v_1,\dots, v_{r}\}$ and $T_2 = \{u_1,\dots, u_{r}\}$, we have\begin{equation*}
    v_1 + \dots +  v_{r} = v = u_1 + \dots +  u_{r},
\end{equation*} which implies that $T_1\cup T_2$ lies on a common $k$-flat. Let\begin{equation*}
    E_{2r} = \bigcup_{v \in [rn]^d}\ \bigcup_{T_1,T_2 \in E_r(v)}  \{T_1, T_2\}.
\end{equation*} Hence, for each $\{T_1, T_2\} \in E_{2r}$, $T_1\cup T_2$ lies on a $k$-flat. Moreover, by Jensen's inequality, we have \begin{equation*}
    |E_{2r}| = \sum_{v \in [rn]^d} \binom{|E_r(v)|}{2} \geq (rn)^d \binom{ \frac{\sum_{v } |E_r(v)| }{ (rn)^d}}{2} = (rn)^d \binom{ |E_r|/ (rn)^d}{2} \geq \frac{|E_r|^2}{4(rn)^d}.
\end{equation*} Since $k$ and $d$ are fixed and $r = \frac{k}{2} + 1$ and $|V|= n^{d - \delta}$,\begin{equation*}
    |E_r|^2  = \binom{|V|}{r}^2 = \binom{|V|}{(k/2) + 1}^2 \geq \Omega(n^{(k + 2)(d-\delta)}).
\end{equation*} Combining the two inequalities above gives \begin{equation*}
    |E_{2r}| \geq \Omega(n^{(k + 1)d - (k + 2)\delta}).
\end{equation*}

We say that  $\{T_1, T_2\} \in E_{2r}$ is \emph{good} if $T_1\cap T_2 = \emptyset$, and the $(k + 2)$-tuple $(T_1\cup T_2)$ is non-degenerate.  Otherwise, we say that $\{T_1,T_2\}$ is \emph{bad}. In what follows, we will show that at least half of the pairs (i.e. elements) in $E_{2r}$ are good. To this end, we will need the following claim.
\begin{claim}\label{first}
If $\{T_1,T_2\}\in E_{2r}$ is bad, then $T_1\cup T_2$ lies on a $(k-1)$-flat.
\end{claim}
\begin{proof}[Proof of Claim]

Write $T_1 = \{v_1,\dots, v_{r}\}$ and $T_2 = \{u_1,\dots, u_{r}\}$. Let us consider the following cases.

\medskip  

\noindent \emph{Case 1.} Suppose $T_1\cap T_2 \neq \emptyset$. Then, without loss of generality, there is an integer $j < r$ such that\begin{equation*}
    v_1 + \dots + v_j = u_1 + \dots + u_j,
\end{equation*} where $v_1,\dots,v_j,u_1,\dots,u_j$ are all distinct elements, and $v_t = u_t$ for $t> j$. Thus $|T_1\cup T_2| = 2j  + (r-j)$. The $2j$ elements above lie on a $(2j - 2)$-flat. Adding the remaining $r-j$ points implies that $T_1\cup T_2$ lies on a $(j-2 + r)$-flat. Since $r = \frac{k}{2} + 1$ and $j \leq \frac{k}{2},$ $T_1\cup T_2$ lies on a $(k-1)$-flat.

\medskip

\noindent \emph{Case 2.} Suppose $T_1\cap T_2 = \emptyset$. Then $T_1\cup T_2$ must be degenerate, which means there is a subset $S\subset T_1\cup T_2$ of $j$ elements such that $S$ lies on a $(j-2)$-flat, for some $3 \leq j \leq k + 1$.  Without loss of generality, we can assume that $v_1 \not\in S$. Hence, $(T_1\cup T_2)\setminus \{v_1\}$ lies on a $(k-1)$-flat. On the other hand, we have\begin{equation*}
    v_1 = u_1+\dots+ u_{r} -v_2 -\dots -  v_{r}.
\end{equation*} Hence, $v_1$ is in the affine hull of $(T_1\cup T_2)\setminus \{v_1\}$ which implies that $T_1\cup T_2$ lies on a $(k-1)$-flat. 
\end{proof}

We are now ready to prove the following claim.
\begin{claim}\label{second}
At least half of the pairs in $E_{2r}$ are good.
\end{claim}
\begin{proof}[Proof of Claim]
For the sake of contradiction, suppose at least half of the pairs in $E_{2r}$ are bad. Let $H$ be the collection of all the $j$-flats spanned by subsets of $V$ for all $j\leq k-1$. Notice that if $S\subset V$ spans a $j$-flat $h$, then $h$ is also spanned by only $j+1$ elements from $S$. So we have\begin{equation*}
    |H| \leq \sum_{j=0}^{k-1}|V|^{j+1} \leq k n^{k(d - \delta)}.
\end{equation*}
For each bad pair $\{T_1, T_2\} \in E_{2r}$, $T_1\cup T_2$ lies on a $j$-flat from $H$ by Claim~\ref{first}.  By the pigeonhole principle, there is a $j$-flat $h$ with $j\leq k-1$ such that at least\begin{equation*}
    \frac{|E_{2r}|/2}{|H|} \geq \frac{\Omega(n^{(k + 1)d - (k + 2)\delta})}{2kn^{k(d - \delta)}}  = \Omega(n^{d - 2\delta })
\end{equation*} bad pairs from $E_{2r}$ have the property that their union lies in $h$.  On the other hand, since $h$ contains at most $n^{k-1}$ points from $[n]^d$, $h$ can correspond to at most $O(n^{(k-1)(k + 2)})$ bad pairs from $E_{2r}$. Since we assumed $d - 2\delta > (k-1)(k + 2)$, we have a contradiction for $n$ sufficiently large.
\end{proof}

Each good pair $\{T_1,T_2\}\in E_{2r}$ gives rise to a non-degenerate $(k + 2)$-tuple $T_1\cup T_2$ that lies on a $k$-flat. On the other hand, any such $(k + 2)$-tuple in $V$ will correspond to at most $\binom{k+2}{r}$ good pairs in $E_{2r}$. Hence, by Claim~\ref{second}, there are at least\begin{equation*}
    \left. \frac{|E_{2r}|}{2}\middle/\binom{k+2}{r}\right.=\Omega(n^{(k + 1)d - (k + 2)\delta})
\end{equation*} non-degenerate $(k + 2)$-tuples that lie on a $k$-flat, concluding the proof.
\end{proof}

In the other direction, we will use the following upper bounds.
\begin{lemma}\label{maxdegree}
For real number $\delta > 0$ and fixed positive integers $d,k,i$, such that $i<k+2$, suppose $U,V\subset [n]^d$ satisfy $|U|=i$ and $|V|=n^{d-\delta}$, then $V$ contains at most $n^{(k+1-i)(d-\delta)+k}$ non-degenerate $(k+2)$-tuples that lie on a $k$-flat and contain $U$.
\end{lemma}
\begin{proof}
If $U$ spans a $j$-flat for some $j<i-1$, then by definition no non-degenerate $(k+2)$-tuple contains $U$. Hence we can assume $U$ spans a $(i-1)$-flat. Observe that a non-degenerate $(k+2)$-tuple $T$, which lies on a $k$-flat and contains $U$, must contain a $(k+1)$-tuple $T'\subset T$ such that $T'$ spans a $k$-flat and $U\subset T'$. Then there are at most $n^{(k + 1 - i)(d-\delta)}$ ways to add $k + 1 - i$ points to $U$ from $V$ to obtain such $T'$. After $T'$ is determined, there are at most $n^k$ ways to add a final point from the affine hull of $T'$ to obtain $T$. So we conclude the proof by multiplication.
\end{proof}

\begin{lemma}\label{count}
For positive integers $\ell \leq d$, the grid $[n]^d$ contains at most $\ell\cdot n^{(\ell+1)d+(s-1)\ell}$ many $(\ell+s)$-tuples that lie on an $\ell$-flat.
\end{lemma}
\begin{proof}
We count the number of ways to choose an $(\ell+s)$-tuple $T$ that spans a $j$-flat. There are at most $n^{(j+1)d}$ ways to choose a subset $T'\subset T$ of size $j+1$ that spans the affine hull of $T$. After this $T'$ is determined, there are at most $n^{(\ell+s-1-j)j}$ ways to add the remaining $\ell+s-1-j$ points from the $j$-flat spanned by $T'$. Then the total number of $(\ell+s)$-tuples that lie on an $\ell$-flat is at most \begin{equation*}
    \sum_{j=1}^{\ell} n^{(j+1)d+(\ell+s-1-j)j}\leq \sum_{j=1}^{\ell} n^{(j+1)d+(\ell+s-1-j)\ell} \leq \sum_{j=1}^{\ell} n^{(\ell+1)d+(s-1)\ell} \leq \ell \cdot n^{(\ell+1)d+(s-1)\ell},
\end{equation*} where the second inequality uses $\ell \leq d$.
\end{proof}

\section{Proof of Theorem~\ref{main}}

In this section, we use the hypergraph container method to prove Theorem \ref{main}. We shall assume basic notions about hypergraphs and follow the strategy outlined in \cite{BS}. Let $\mathcal{H}=(V(\mathcal{H}),E(\mathcal{H}))$ denote a $r$-uniform hypergraph. For any $U\subset V(\mathcal{H})$, its \textit{degree} is the number of edges containing $U$. For each $i \in [r]$, we use $\Delta_{i}(\mathcal{H})$ to denote the maximum degree among all $U$ of size $i$. For $S \subset V(\mathcal{H})$, we use $\mathcal{H}[S]$ to denote the \textit{induced sub-hypergraph} on $S$. We shall use the following version of the hypergraph container lemma, which is Theorem 4.2 in~\cite{morris2016534}.

\begin{lemma}\label{HCL}
Let $r \ge 2$ be an integer and $c>0$ be sufficiently small with respect to $r$. If $\mathcal{H}=(V,E)$ is an $r$-uniform hypergraph and $0<\tau<1/2$ is a real number such that \begin{equation*}
        \Delta_i(\mathcal{H}) \leq c \cdot \tau^{i-1} \frac{|E|}{|V|} \quad\text{for all $2\leq i\leq r$,}
\end{equation*} then there exists a family $\mathcal{C}$ of vertex subsets of $\mathcal{H}$ with the following properties:
\begin{itemize}
            \item[(a)] Every independent set of $\mathcal{H}$ is contained in some $C \in \mathcal{C}$.
            \item[(b)] $|\mathcal{C}| \leq \exp\left(c^{-1} \cdot \tau |V| \cdot \log(1/\tau)\right)$.
            \item[(c)] For every $C \in \mathcal{C}$, we have $|E(\mathcal{H}[C])| \leq (1 - c)|E|$.
\end{itemize}
\end{lemma}

The main result of this section is the following theorem.
\begin{theorem}\label{maink}
Let $k,\ell$ be fixed integers such that $\ell \geq k\geq 2$ and $k$ is even. Then for any $\epsilon > 0$, there is a constant $d= d(\epsilon,k,\ell)$ such that the following holds. For infinitely many values of $N$, there is a set $V$ of $N$ points in $\mathbb{R}^{d}$ such that no $\ell+3$ members of $V$ lie on an $\ell$-flat, and every subset of $V$ without $k+2$ members on a $k$-flat has size at most $O\left(N^{\frac{\ell + 2}{2(k + 1)}  + \epsilon}\right)$.  
\end{theorem}

Before we prove Theorem~\ref{maink}, let us show that it implies Theorem~\ref{main}.

\begin{proof}[Proof of Theorem \ref{main}]
In dimensions $d'\geq 3$ where $d'$ is odd, we apply Theorem \ref{maink} with $k = \ell = d' - 1$ to obtain a point set $V$ of size $N$ in $\mathbb{R}^d$ with the property that no $d' +2$ members lie on a $(d' - 1)$-flat, and every subset of size $\Omega\left( N^{\frac{1}{2} + \frac{1}{2d'} + \epsilon} \right)$ contains $d' + 1$ members on a $(d' -1)$-flat.  By projecting $V$ to a generic $d'$-dimensional subspace of $\mathbb{R}^d$, we obtain $N$ points in $\mathbb{R}^{d'}$ with no $d' + 2$ members on a common hyperplane, and every subset in general position has size $O\left( N^{\frac{1}{2} + \frac{1}{2d'} + \epsilon} \right)$.

In dimensions $d' \geq 4$ where $d'$ is even, we apply Theorem \ref{maink} with $k = d'- 2$ and $\ell  = d' -1$ to obtain a point set $V$ of size $N$ in $\mathbb{R}^d$ with the property that no $d' +2$ members on a $(d'-1)$-flat, and every subset of size $\Omega\left( N^{\frac{1}{2} + \frac{1}{d' - 1} + \epsilon} \right)$ contains $d'$ members on a $(d' -2)$-flat. By adding another point from this subset, we obtain $d' + 1$ members on a $(d' - 1)$-flat. Hence, by projecting to $V$ a generic $d'$-dimensional subspace of $\mathbb{R}^d$, we obtain $N$ points in $\mathbb{R}^{d'}$ with no $d' + 2$ members on a common hyperplane, and every subset in general position has size $O\left( N^{\frac{1}{2} + \frac{1}{d' - 1} + \epsilon} \right)$.

Since $\epsilon$ is arbitrary and $N$ grows to infinity, we can conclude the proof of Theorem~\ref{main} after renaming $d'$ to $d$.
\end{proof}

\begin{proof}[Proof of Theorem \ref{maink}]

Let $d$ be a sufficiently large integer and $n$ tend to infinity. We denote $\mathcal{H}$ as the hypergraph with $V(\mathcal{H})=[n]^d$ and $E(\mathcal{H})$ consisting of non-degenerate $(k+2)$-tuples $T$ such that $T$ lies on a $k$-flat. We shall construct a rooted tree $\mathfrak{T}$ whose nodes are labelled with vertex subsets of $\mathcal{H}$ as follows. We start with $\mathfrak{T}$ consisting of one root node labelled with $V(\mathcal{H})$. Iteratively, if there is a leaf $x \in \mathfrak{T}$ whose labelled set $C_x$ has size at least $n^{\frac{k}{k + 1}d + k}$, we apply Lemma~\ref{HCL} to $\mathcal{H}[C_x]$ with $\tau = n^{-\frac{k}{k + 1}d  + \delta +  \epsilon}$ where $\delta$ is defined by $|C_x| = n^{d - \delta}$. As a consequence, Lemma~\ref{HCL} produces a collection $\mathcal{C}$ of subsets of $C_x$. Then we create a child of $x$ in $\mathfrak{T}$ labelled by $C$ for each $C \in \mathcal{C}$. The iteration continues until there is no leaf $x\in \mathfrak{T}$ with $|C_x| \geq n^{\frac{k}{k + 1}d + k}$.

During the interative construction of $\mathfrak{T}$, we need to verify the hypothesis of Lemma~\ref{HCL}, that is, \begin{equation*}
    \Delta_i(\mathcal{H}[C_x]) \leq c \cdot \tau^{i-1} \frac{|E(\mathcal{H}[C_x])|}{|V(\mathcal{H}[C_x])|} \quad\text{for all $2\leq i\leq k+2$.}
\end{equation*} To check this, we use Lemma~\ref{maxdegree} to upper bound $\Delta_i(\mathcal{H}[C_x])$ for $2\leq i < k+2$, and use the trivial bound $\Delta_i(\mathcal{H}[C_x])\leq 1$ for $i = k+2$. On the other hand, we use Lemma~\ref{supersaturation} to lower bound $|E(\mathcal{H}[C_x])|$. We shall use $n^{d - \delta} = |V(\mathcal{H}')| \geq n^{\frac{k}{k + 1}d + k}$ as well. Since this is a straightforward computation, whose detail will be given as Claim~\ref{tauapp} in the proof of Theorem~\ref{thm_remark}, we skip it here.

Now, we analyze this rooted tree $\mathfrak{T}$. According to Lemma~\ref{HCL}(c), if $y$ (labelled with $C_y$) is a child of $x$ (labelled with $C_x$) in $\mathfrak{T}$, the number of edges induced by $C_y$ shrinks from that by $C_x$ by a constant factor $(1-c)$. On the other hand, a reasonably large set induces many edges in $\mathcal{H}$ by Lemma~\ref{supersaturation} (assuming $d$ is large). This means the height of $\mathfrak{T}$ is upper bounded by $O(\log n)$, and in particular our iterative construction ends. According to Lemma~\ref{HCL}(b), the number of children of any node $x$ in $\mathfrak{T}$ is at most \begin{equation*}
    |\mathcal{C}| \leq \exp({c^{-1} \cdot \tau |C_x| \cdot \log(1/\tau)}) \leq \exp\left(O\left(n^{\frac{d}{k + 1} + \epsilon}\cdot \log{n}\right)\right).
\end{equation*} Therefore, let $\mathfrak{C}$ be the collection of sets labelling the leaves of $\mathfrak{T}$.  Hence, we have \begin{equation*}
    |\mathfrak{C}| \leq \exp\left(O\left(n^{\frac{d}{k + 1} + \epsilon}\cdot \log^2{n}\right)\right) \quad\text{and}\quad |C| \leq n^{\frac{k}{k + 1}d + k} ~\text{for all $C\in \mathfrak{C}$}.
\end{equation*} Furthermore, if $I$ is an independent set of $\mathcal{H}$ that is contained in a vertex subset $C_x$ labelling a non-leaf node $x$, then by  the construction of $\mathfrak{T}$ and Lemma~\ref{HCL}(a), there exists a child $y$ of $x$ in $\mathfrak{T}$ whose labelling set $C_y$ contains $I$. This implies every independent set of $\mathcal{H}$ is contained in some member of $\mathfrak{C}$. Elements in this collection $\mathfrak{C}$ are called containers.

Next, we randomly select a subset of $[n]^d$ by keeping each point independently with probability $p$. Let $S$ be the set of selected elements. Then for each $(\ell + 3)$-tuple $T$ in $S$ that lies on an $\ell$-flat, we delete one point from $T$. We denote the resulting set of points by $S'$. By Lemma~\ref{count}, we have \begin{equation*}
   \mathbb{E}[|S'|] \geq pn^d - p^{\ell+3}\ell n^{(\ell+1)d+2\ell}.
\end{equation*} By setting $p=(2\ell)^{-\frac{1}{\ell+2}}n^{-\frac{\ell}{\ell+2}(d+2)}$, we have
\begin{equation*}
   \mathbb{E}[|S'|] \geq \frac{pn^d}{2} =\Omega\left(n^{\frac{2(d-\ell)}{\ell+2}}\right).
\end{equation*}

Finally, we set $m = n^{\frac{d}{k + 1} + 2\epsilon}$. Let $X$ denote the number of independent sets of $\mathcal{H}$ in $S'$ with cardinality $m$. Using the family of containers, we have \begin{align*}
    \mathbb{E}[X] & \leq  |\mathfrak{C}|\cdot \binom{n^{\frac{k}{k + 1}d + k}}{m} \cdot p^{m} \\
    &\leq  \exp\left(O\left(n^{\frac{d}{k + 1} + \epsilon}\cdot \log^2{n}\right)\right) \cdot \left(\frac{e \cdot n^{\frac{k}{k + 1}d + k}}{m}\right)^m p^m \\ 
    &\leq  \exp\left(O\left(n^{\frac{d}{k + 1} + \epsilon}\cdot \log^2{n}\right)\right) \cdot \left(e \cdot n^{\frac{k-1}{k + 1}d + k -2\epsilon}\right)^m \left((2\ell)^{-\frac{1}{\ell+2}} \cdot n^{-\frac{\ell}{\ell+2}(d+2)}\right)^m \\
    &\leq \exp\left(O\left(n^{\frac{d}{k + 1} + \epsilon}\cdot \log^2{n}\right)\right) \cdot \left(\frac{1}{2}\right)^m\\
    &\leq o(1).
\end{align*} Here, the fourth inequality uses the following consequence of $k\leq \ell$ and $d$ being large:\begin{equation*}
    \frac{k-1}{k + 1}d + k -2\epsilon-\frac{\ell}{\ell+2}(d+2) < 0.
\end{equation*}

Notice that $|S'|$ is exponentially concentrated around its mean by Chernoff's inequality. Therefore, some realization of $S'$ satisfies: $|S'|=N=\Omega(n^{2(d-\ell)/(\ell + 2)})$; $S'$ contains no $(\ell+3)$-tuples on a $\ell$-flat; and $\mathcal{H}[S']$ does not contain an independent set of $\mathcal{H}$ with cardinality \begin{equation*}
    m = n^{\frac{d}{k + 1} + 2\epsilon} = O\left( N^{\frac{\ell + 2}{2(k + 1)} + \frac{(\ell+2)\ell}{2(k+1)(d-\ell)} + \frac{\ell + 2}{d-\ell}\epsilon}\right) \leq O\left( N^{\frac{\ell + 2}{2(k + 1)} + \epsilon}\right).
\end{equation*} Here, we assume $d = d(\epsilon,k,\ell)$ is sufficiently large so that\begin{equation*}
    \frac{(\ell+2)\ell}{2(k+1)(d-\ell)} + \frac{\ell + 2}{d-\ell}\epsilon \leq \epsilon.
\end{equation*} Notice that $S'$ not containing an independent set of size $m$ means every subset of $S'$ of size $m$ contains $k+2$ points on a $k$-flat. We conclude the proof by renaming $S'$ to $V$.
\end{proof}

\section{Proof of Theorem~\ref{thm_remark}}

In this section, we prove Theorem \ref{thm_remark}.  The proof is essentially the same as in the previous section with a different choice of parameters.  For the reader's convenience, we include the details here. We start by proving the following theorem.

\begin{theorem}\label{mainapp}
Let $k,\ell,s$ be fixed integers such that $\ell\geq k\geq 2$, $s\geq 2$, $k$ is even, and $\frac{2\ell+s-1}{\ell+s-1} < \frac{2k}{k+1}$. Then for any $\epsilon > 0$, there is a constant $d = d(\epsilon,k,\ell,s)$ such that the following holds. For infinitely many values of $N$, there is a set $V$ of $N$ points in $\mathbb{R}^{d}$ such that no $\ell+s$ members of $V$ lie on an $\ell$-flat, and every subset of $V$ without $k + 2$ members on a $k$-flat has size at most $O\left( N^{\frac{1}{2}+\epsilon} \right)$.
\end{theorem}

\begin{proof}

Just as before, let $\mathcal{H}$ be the hypergraph with $V(\mathcal{H})=[n]^d$ and $E(\mathcal{H})$ consisting of non-degenerate $(k+2)$-tuples $T$ such that $T$ lies on a $k$-flat. We let $q=q(k,r,s)$ be a quantity that will be determined later. We again construct a rooted tree $\mathfrak{T}$ whose nodes are labelled with vertex subsets of $\mathcal{H}$. We start with $\mathfrak{T}$ consisting of one root node labelled with $V(\mathcal{H})$. Iteratively, if there is a leaf $x \in \mathfrak{T}$ whose labelled set $C_x$ has size at least $n^{qd + k}$, we apply Lemma~\ref{HCL} to $\mathcal{H}' = \mathcal{H}[C_x]$ with $\tau = n^{-qd  + \delta +  \epsilon}$ where $\delta$ is defined by $|C_x| = n^{d - \delta}$. We shall use the claim below to verify the hypothesis of Lemma~\ref{HCL}. As a consequence, Lemma~\ref{HCL} produces a collection $\mathcal{C}$ of subsets of $C_x$. Then we create a child of $x$ in $\mathfrak{T}$ labelled by $C$ for each $C \in \mathcal{C}$. The iteration continues until there is no leaf $x\in \mathfrak{T}$ with $|C_x| \geq n^{qd + k}$.

\begin{claim}\label{tauapp}
If $\frac{1}{2} < q \leq \frac{k}{k+1}$ and $\mathcal{H}'$ defined as above, then \begin{equation*}
    \Delta_i(\mathcal{H}') \leq c \cdot \tau^{i-1} \frac{|E(\mathcal{H}')|}{|V(\mathcal{H}')|} \quad\text{for all $2\leq i\leq k+2$,}
\end{equation*} where $c$ is the constant in Lemma~\ref{HCL} depending only on $k$.
\end{claim}
\begin{proof}[Proof of Claim]
First, we notice that \begin{equation}\label{tauapp_eq}
    n^{d - \delta} = |V(\mathcal{H}')| \geq n^{qd + k} \implies \delta \leq d - qd - k.
\end{equation} Assuming $d$ is large, we have $|E(\mathcal{H}')| \geq \Omega(n^{(k + 1)d - (k + 2)\delta})$ by Lemma~\ref{supersaturation}.

For $2\leq i < k+2$, Lemma~\ref{maxdegree} gives us $\Delta_i(\mathcal{H}') \leq n^{(k+1-i)(d-\delta)+k}$. Hence, it suffices to check \begin{equation*}
    n^{(k+1-i)(d-\delta) + k} \ll \left(n^{-qd+\delta+\epsilon}\right)^{i-1} \cdot \frac{n^{(k+1)d-(k+2)\delta}}{n^{d-\delta}}.
\end{equation*} Simplifying and comparing the exponents over $n$, this is implied by \begin{equation*}
    (i-1) d + k + (i-1)\epsilon > (i-1) q d + \delta.
\end{equation*} Since $d$ is sufficiently large, it suffices to compare the coefficients of $d$. Applying \eqref{tauapp_eq} and simplifying the terms, the inequality above is implied by $i-1 \geq (i-2)q + 1$, which is true by our hypothesis.

For $i = k+2$, we have $\Delta_i(\mathcal{H}') \leq 1$ trivially. Hence, it suffices to check \begin{equation*}
    1 \ll \left(n^{-qd+\delta+\epsilon}\right)^{k+1} \cdot \frac{n^{(k+1)d-(k+2)\delta}}{n^{d-\delta}}.
\end{equation*} Simplifying and comparing the exponents over $n$, this is implied by \begin{equation*}
    (k+1)qd < kd + (k+1)\epsilon.
\end{equation*} Again, since $d$ is sufficiently large, it suffices to compare the coefficients of $d$. The inequality above is implied by $(k+1) q \leq k$, which is true by our hypothesis.
\end{proof}

We can analyze this rooted tree $\mathfrak{T}$ using arguments similar to the previous section. We can conclude that there exists a collection $\mathfrak{C}$ of vertex subsets of $\mathcal{H}$ with \begin{equation*}
    |\mathfrak{C}| \leq \exp\left(O\left(n^{d-qd + \epsilon}\cdot \log^2{n}\right)\right) \quad\text{and}\quad |C| \leq n^{qd + k} ~\text{for all $C\in \mathfrak{C}$}.
\end{equation*} and every independent set of $\mathcal{H}$ is contained in some member of $\mathfrak{C}$.

Next, we randomly select a subset of $[n]^d$ by keeping each point independently with probability $p$. Let $S$ be the set of selected elements. Then for each $(\ell + s)$-tuple $T$ in $S$ that lies on an $\ell$-flat, we delete one point from $T$. We denote the resulting set of points by $S'$. By Lemma~\ref{count}, we have \begin{equation*}
   \mathbb{E}[|S'|] \geq pn^d - p^{\ell+s}\ell n^{(\ell+1)d+(s-1)\ell}.
\end{equation*} By setting $p=(2\ell)^{-\frac{1}{\ell+s-1}}n^{-\frac{\ell}{\ell+s-1}(d+s-1)}$, we have
\begin{equation*}
   \mathbb{E}[|S'|] \geq \frac{pn^d}{2} =\Omega\left(n^{\frac{(s-1)(d-\ell)}{\ell+s-1}}\right).
\end{equation*}

Finally, we set $m = n^{d-qd + 2\epsilon}$. Let $X$ denote the number of independent sets of $\mathcal{H}$ in $S'$ with cardinality $m$. With a foresight soon to be self-evident, we choose \begin{equation}\label{qchoice}
    q = \frac{1}{2} \cdot \frac{2\ell+s-1}{\ell+s-1} + \frac{1}{2d} \cdot \frac{\ell(s-1)}{\ell+s-1} -\frac{k}{2d}.
\end{equation} We remark that our hypothesis on $k,\ell,s$ implies $\frac{1}{2}< q \leq \frac{k}{k + 1}$ assuming $d$ is large, hence Claim~\ref{tauapp} can be applied in construction of $\mathfrak{T}$. 

Using the family $\mathfrak{C}$, we can estimate \begin{align*}
    \mathbb{E}[X] & \leq  |\mathfrak{C}|\cdot \binom{n^{qd + k}}{m} \cdot p^{m} \\
    &\leq  \exp\left(O\left(n^{d-qd + \epsilon}\cdot \log^2{n}\right)\right) \cdot \left(\frac{e \cdot n^{qd + k}}{m}\right)^m p^m \\ 
    &\leq  \exp\left(O\left(n^{d-qd + \epsilon}\cdot \log^2{n}\right)\right) \cdot \left(e \cdot n^{(2q-1)d + k -2\epsilon}\right)^m \left((2\ell)^{-\frac{1}{\ell+s-1}}n^{-\frac{\ell}{\ell+s-1}(d+s-1)}\right)^m \\
    &\leq \exp\left(O\left(n^{d-qd + \epsilon}\cdot \log^2{n}\right)\right) \cdot \left(\frac{1}{2}\right)^m\\
    &\leq o(1).
\end{align*} Here, the fourth inequality uses the following consequence of \eqref{qchoice}: \begin{equation*}
    (2q-1)d + k -2\epsilon -\frac{\ell}{\ell+s-1}(d+s-1) < 0.
\end{equation*}

Notice that $|S'|$ is exponentially concentrated around its mean by Chernoff's inequality. Therefore, some realization of $S'$ satisfies: $|S'|=N=\Omega\left(n^{\frac{(s-1)(d-\ell)}{\ell+s-1}}\right)$; $S'$ contains no $(\ell+s)$-tuples on a $\ell$-flat; and $\mathcal{H}[S']$ does not contain an independent set of $\mathcal{H}$ with cardinality \begin{equation*}
    m = n^{d - qd + 2\epsilon} = O\left( N^{\frac{1}{2}+ \left(\frac{k}{2} + 2\epsilon\right)\cdot\frac{\ell+s-1}{(s-1)(d-\ell)}} \right) \leq O\left( N^{\frac{1}{2} + \epsilon}\right).
\end{equation*} Here, we assume $d = d(\epsilon,k,\ell,s)$ is sufficiently large so that\begin{equation*}
    \left(\frac{k}{2} + 2\epsilon\right)\cdot\frac{\ell+s-1}{(s-1)(d-\ell)} \leq \epsilon.
\end{equation*} Since $S'$ does not contain an independent set of size $m$, every subset of $S'$ of size $m$ contains $k+2$ points on a $k$-flat. We conclude the proof by renaming $S'$ to $V$.
\end{proof}

\begin{proof}[Proof of Theorem \ref{thm_remark}]

In dimensions $d'\geq 3$ where $d'$ is odd, we obtain an upper bound for $\alpha_{d',s'}(N)$ with $d's' + 2 > 2d' + 2s'$. We set $k = \ell = d'-1$ and $s = s' + 1$, so we can verify $\frac{2\ell+s-1}{\ell+s-1} < \frac{2k}{k+1}$. Hence we can apply Theorem~\ref{mainapp} to obtain a point set $V$ of size $N$ in $\mathbb{R}^{d}$ with the property that no $d' +s'$ members lie on a $(d' - 1)$-flat, and every subset of size $\Omega(N^{\frac{1}{2} + \epsilon})$ contains $d' + 1$ members on a $(d' -1)$-flat. By projecting $V$ to a generic $d'$-dimensional subspace of $\mathbb{R}^d$, we obtain $N$ points in $\mathbb{R}^{d'}$ with no $d' + s'$ members on a common hyperplane, and every subset in general position has size $O( N^{\frac{1}{2} +  \epsilon})$.

In dimensions $d' \geq 4$ where $d'$ is even, we obtain an upper bound for $\alpha_{d',s'}(N)$ with $d's' + 2 > 2d' + 3s'$. We set $k = d'- 2$, $\ell  = d' -1$, and $s  = s' + 1$, so we can verify $\frac{2\ell+s-1}{\ell+s-1} < \frac{2k}{k+1}$. Hence we can apply Theorem~\ref{mainapp} to obtain a point set $V$ of size $N$ in $\mathbb{R}^d$ with the property that no $d' +s'$ members on a $(d'-1)$-flat, and every subset of size $\Omega(N^{\frac{1}{2} + \epsilon})$ contains $d'$ members on a $(d' -2)$-flat. By adding another point from this subset, we obtain $d' + 1$ members on a $(d' - 1)$-flat. Hence, by projecting to $V$ a generic $d'$-dimensional subspace of $\mathbb{R}^d$, we obtain $N$ points in $\mathbb{R}^{d'}$ with no $d' + s'$ members on a common hyperplane, and every subset in general position has size $O( N^{\frac{1}{2} +  \epsilon})$.

Since $\epsilon$ is arbitrary and $N$ grows to infinity, we can conclude the proof of Theorem~\ref{thm_remark} after renaming $d'$ to $d$ and $s'$ to $s$.
\end{proof}

\section{Proof of Theorem~\ref{main2}}

In this section, we will give a proof of Theorem~\ref{main2}. Let $V\subset [n]^d$ such that there are no $k +2$ points that lie on a $k$-flat. In \cite{L}, Lefmann showed that $|V| \leq O\left(n^{\frac{d}{\lfloor (k + 2)/2\rfloor}}\right)$. To see this, assume that $k$ is even and consider all elements of the form $v_1 + \dots  + v_{\frac{k}{2} + 1}$, where $v_i\neq v_j$ and $v_i \in V$.  All of these elements are distinct, since otherwise we would have $k + 2$ points on a $k$-flat. In other words, the equation \begin{equation*}
\left(\textbf{x}_1 + \dots + \textbf{x}_{\frac{k}{2} + 1}\right)  - \left(\textbf{x}_{\frac{k}{2} + 2} + \dots + \textbf{x}_{k + 2}\right)  = \textbf{0},
\end{equation*} does not have a solution with $\{\textbf{x}_1 , \dots , \textbf{x}_{\frac{k}{2} + 1}\}$ and $\{\textbf{x}_{\frac{k}{2} + 2} , \dots , \textbf{x}_{k + 2}\}$ being two different $(\frac{k}{2} + 1)$-tuples of $V$. Therefore, we have $\binom{|V|}{\frac{k}{2} + 1} \leq (kn)^d$, and this implies Lefmann's bound.  

More generally, let us consider the equation \begin{equation}\label{geneqn}
    c_1\textbf{x}_1+c_2\textbf{x}_2+\dots+c_r\textbf{x}_r=\textbf{0},
\end{equation} with constant coefficients $c_i \in \mathbb{Z}$ and $\sum_i c_i = 0$. Here, the variables $\textbf{x}_i$ takes value in $\mathbb{Z}^d$. A solution $(\textbf{x}_1,\dots, \textbf{x}_r)$ to equation (\ref{geneqn}) is called \emph{trivial} if there is a partition $\mathcal{P}: [r] = \mathcal{I}_1\cup \dots \cup \mathcal{I}_t$, such that $\textbf{x}_j = \textbf{x}_{\ell}$ if and only if $j,\ell\in \mathcal{I}_i$, and $\sum_{j \in \mathcal{I}_i} c_j = 0$ for all $i\in [t]$. In other words, being trivial means that, after combining like terms, the coefficient of each $\textbf{x}_i$ becomes zero. Otherwise, we say that the solution $(\textbf{x}_1,\dots, \textbf{x}_r)$ is \emph{non-trivial}. A natural extremal problem is to determine the maximum size of a set $A \subset [n]^d$ with only trivial solutions to (\ref{geneqn}).  When $d = 1$, this is a classical problem in additive number theory, and we refer the interested reader to \cite{R,OB,LV,CT}.

By combining the arguments of Cilleruelo and Timmons \cite{CT} and Jia \cite{J}, we establish the following theorem.
\begin{theorem}\label{multifold}
Let $d,r$ be fixed positive integers. Suppose $V\subset [n]^d$ has only trivial solutions to each equation of the form\begin{equation}\label{ineqn}c_1\left((\textbf{x}_1 + \dots + \textbf{x}_r) - (\textbf{x}_{r + 1} + \dots + \textbf{x}_{2r})\right)  = c_2\left((\textbf{x}_{2r + 1} + \dots + \textbf{x}_{3r}) - (\textbf{x}_{3r+1} + \dots + \textbf{x}_{4r})\right),
\end{equation} for integers $c_1,c_2$ such that $1 \leq c_1,c_2 \leq n^{\frac{d}{2rd + 1}}$. Then we have\begin{equation*}
    |V| \leq   O\left(n^{\frac{d}{2r}\left(1 - \frac{1}{2rd + 1}\right)}\right).
\end{equation*}
\end{theorem}

Notice that Theorem~\ref{main2} follows from Theorem~\ref{multifold}. Indeed, when $k+2$ is divisible by $4$, we set $r=(k+2)/4$. If $V\subset [n]^d$ contains $k + 2$ points $\{v_1,\dots, v_{k + 2}\}$ that is a non-trivial solution to \eqref{ineqn} with $\textbf{x}_i = v_i$, then $\{v_1,\dots, v_{k + 2}\}$ must lie on a $k$-flat.  Hence, when $k + 2$ is divisible by $4$, we have\begin{equation*}
    a(d,k,n) \leq   O\left(n^{\frac{d}{(k + 2)/2}\left(1 - \frac{1}{(k + 2)d/2 + 1}\right)}\right).
\end{equation*} Since we have $a(d,k,n) < a(d,k - 1,n)$, this implies that for all $k\geq 2$, we have\begin{equation*}
    a(d,k,n) \leq O\left(n^{\frac{d}{2\lfloor(k + 2)/4\rfloor}\left(1  - \frac{1}{2\lfloor(k + 2)/4\rfloor d  + 1}\right)}\right).
\end{equation*}

In the proof of Theorem~\ref{multifold}, we need the following well-known lemma (see e.g. Lemma~2.1 in \cite{CT} and Theorem~4.1 in \cite{R}). For $U,T \subset \mathbb{Z}^d$ and $x \in \mathbb{Z}^d$, we define\begin{equation*}
    \Phi_{U - T}(x) = \{(u,t): u - t = x, u \in U, t \in T\}.
\end{equation*}
\begin{lemma}\label{CS}
For finite sets $U, T \subset \mathbb{Z}^d$, we have\begin{equation*}
    \frac{(|U||T|)^2}{|U+T|} \leq  \sum_{x \in \mathbb{Z}^d} |\Phi_{U-U}(x)|\cdot|\Phi_{T-T}(x)|.
\end{equation*}
\end{lemma}

\begin{proof}[Proof of Theorem~\ref{multifold}]

Let $d$, $r$, and $V$ be as given in the hypothesis. Let $m \geq 1$ be an integer that will be determined later.  We define \begin{equation*}
    S_r = \{v_1 + \dots +v_{r}: v_i \in V, v_i \neq v_j\},
\end{equation*} and a function \begin{equation*}
    \sigma:\binom{V}{r}\rightarrow S_r,\ \{v_1,\dots, v_r\} \mapsto v_1 + \dots + v_r.
\end{equation*} Notice that $\sigma$ is a bijection. Indeed, suppose on the contrary that \begin{equation*}
    v_1  + \dots + v_{r} = v'_1 + \dots + v'_{r}
\end{equation*} for two different $r$-tuples in $V$. Then by setting $(\textbf{x}_1,\dots,\textbf{x}_r)=(v_1,\dots,v_r)$, $(\textbf{x}_{r+1},\dots,\textbf{x}_{2r})=(v'_1,\dots,v'_r)$, $(\textbf{x}_{2r+1},\dots,\textbf{x}_{3r})=(\textbf{x}_{3r+1},\dots,\textbf{x}_{4r})$ arbitrarily, and $c_1=c_2=1$, we obtain a non-trivial solution to \eqref{ineqn}, which is a contradiction. In particular, we have $|S_r| = \binom{|V|}{r}$.

For $j \in [m]$ and $w \in \mathbb{Z}_j^d$, we let\begin{equation*}
    U_{j,w} = \{u \in \mathbb{Z}^d: ju + w \in S_r\}.
\end{equation*} Notice that for fixed $j \in [m]$, we have\begin{equation*}
    \sum_{w \in \mathbb{Z}_j^d} |U_{j,w}| = \sum_{w \in  \mathbb{Z}_j^d} |\{v \in S_r : v \equiv w \text{ mod $j$}\}|  = |S_r|.
\end{equation*} Applying Jensen's inequality to above, we have\begin{equation}\label{sqsum}
\sum_{w \in\mathbb{Z}_j^d} |U_{j,w}|^2 \geq |S_r|^2/j^d.
\end{equation}

For $i \geq 0$, we define\begin{equation*}
    \Phi^i_{U_{j,w}-U_{j,w}}(x)  = \{(u_1,u_2)\in \Phi_{U_{j,w}-U_{j,w}}(x): |\sigma^{-1}(ju_1+w)\cap\sigma^{-1}(ju_2+w)| = i\}.
\end{equation*} It's obvious that these sets form a partition of $\Phi_{U_{j,w}-U_{j,w}}(x)$. We also make the following claims.
\begin{claim}\label{up1}
For a fixed $x\in \mathbb{Z}^d$, we have
\begin{equation*}
\sum_{j \in [m]}\sum_{w \in\mathbb{Z}_j^d } |\Phi^0_{U_{j,w}-U_{j,w}}(x)| \leq 1,
\end{equation*}
\end{claim}
\begin{proof}
For the sake of contradiction, suppose the summation above is at least two, then we have $(u_1,u_2)\in \Phi^0_{U_{j,w}-U_{j,w}}(x)$ and $(u_3,u_4)\in \Phi^0_{U_{j',w'}-U_{j',w'}}(x)$ such that either $(u_1,u_2)\neq(u_3,u_4)$ or $(j,w)\neq (j',w')$.

Let $s_1,s_2,s_3,s_4 \in S_r$ such that $s_1 = ju_1 + w$, $s_2 = ju_2 + w$, $s_3 = j'u_3 + w'$, $s_4 = j'u_4 + w'$ and write $\sigma^{-1}(s_i)=\{v_{i,1},\dots,v_{i,r}\}$. Notice that $u_1 - u_2 = x = u_3 - u_4$. Putting these equations together gives us\begin{equation}\label{check}
    j'((v_{1,1} + \dots + v_{1,r}) - (v_{2,1} + \dots + v_{2,r})) = j((v_{3,1} + \dots + v_{3,r}) - (v_{4,1} + \dots + v_{4,r})).
\end{equation} It suffices to show that \eqref{check} can be seem as a non-trivial solution to \eqref{ineqn}. The proof now falls into the following cases.

\medskip
\noindent \emph{Case 1.}  Suppose $j \neq j'$. Without loss of generality we can assume $j'>j$. Notice that $(u_1,u_2)\in \Phi^0_{U_{j,w}-U_{j,w}}(x)$ implies\begin{equation*}
    \{v_{1,1},\dots,v_{1,r}\}\cap \{v_{2,1},\dots,v_{2,r}\}=\emptyset.
\end{equation*} Then after combining like terms in \eqref{check}, the coefficient of $v_1^1$ is at least $j'-j$, which means this is indeed a non-trivial solution to \eqref{ineqn}.

\medskip
\noindent \emph{Case 2.}  Suppose $j = j'$, then we must have $s_1 \neq s_3$. Indeed, if $s_1=s_3$, we must have $w=w'$ (as $s_1$ modulo $j$ equals $s_3$ modulo $j'$) and $s_2=s_4$ (as $j'(s_1-s_2)=j(s_3-s_4)$). This is a contradiction to either $(u_1,u_2)\neq(u_3,u_4)$ or $(j,w)\neq (j',w')$.

Given $s_1 \neq s_3$, we can assume, without loss of generality, $v_{1,1}\not\in \{v_{3,1},\dots,v_{3,r}\}$. Again, we have $\{v_{1,1},\dots,v_{1,r}\}\cap \{v_{2,1},\dots,v_{2,r}\}=\emptyset$. Hence, after combining like terms in \eqref{check}, the coefficient of $v_1^1$ is positive and we have a non-trivial solution to \eqref{ineqn}.\end{proof}

\begin{claim}\label{up2}
For a finite set $T \subset \mathbb{Z}^d$, and fixed integers $i,j\geq 1$, we have
\begin{equation*}
    \sum_{w\in \mathbb{Z}_j^d}\sum_{x\in \mathbb{Z}^d} |\Phi^{i}_{U_{j,w}-U_{j,w}}(x)|\cdot|\Phi_{T-T}(x)|\leq |V|^{2r-i}|T|.
\end{equation*}
\end{claim}
\begin{proof}
The summation on the left-hand side counts all (ordered) quadruples $(u_1,u_2,t_1,t_2)$ such that $(u_1,u_2)\in \Phi^{i}_{U_{j,w}-U_{j,w}}(t_1-t_2)$. For each such a quadruple, let $s_1,s_2 \in S_r$ such that\begin{equation*}
    s_1 = ju_1 + w \text{\quad and\quad} s_2 = ju_2 + w.
\end{equation*} There are at most $|V|^{2r-i}$ ways to choose a pair $(s_1,s_2)$ satisfying $|\sigma^{-1}(s_1)\cap \sigma^{-1}(s_2)|=i$. Such a pair $(s_1,s_2)$ determines $(u_1,u_2)$ uniquely. Moreover, $(s_1,s_2)$ also determines the quantity\begin{equation*}
    t_1-t_2=u_1-u_2=\frac{s_1-w}{j}-\frac{s_2-w}{j}=\frac{1}{j}(s_1-s_2).
\end{equation*} After such a pair $(s_1,s_2)$ is chosen, there are at most $|T|$ ways to choose $t_1$ and this will also determine $t_2$. So we conclude the claim by multiplication.\end{proof} 

Now, we set $T = \mathbb{Z}_\ell^d$ for some integer $\ell$ to be determined later. Notice that $U_{j,w} + T \subset \{0,1,\dots, \lfloor rn/j\rfloor + \ell-1\}^d$, which implies
\begin{equation}\label{ujwt}
    |U_{j,w} + T|\leq (rn/j+\ell)^d.
\end{equation}

By Lemma~\ref{CS}, we have \begin{equation*}
    \frac{|U_{j,w}|^2||T|^2}{|U_{j,w} + T|} \leq  \sum_{x \in \mathbb{Z}^d}|\Phi_{U_{j,w}-U_{j,w}}(x)|\cdot|\Phi_{T-T}(x)|.
\end{equation*} Summing over all $j \in [m]$ and $w \in \mathbb{Z}_j^d$, and using Claims~\ref{up1}~and~\ref{up2}, we can compute\begin{align*}
\sum_{j\in[m]}\sum_{w \in \mathbb{Z}_j^d}\frac{|U_{j,w}|^2||T|^2}{|U_{j,w} + T|}  & \leq   \sum_{j\in[m]}\sum_{w \in \mathbb{Z}_j^d} \sum_{x \in \mathbb{Z}^d}|\Phi_{U_{j,w}-U_{j,w}}(x)|\cdot|\Phi_{T-T}(x)| \\
&  =    \sum_{x \in \mathbb{Z}^d} \sum_{j\in[m]}\sum_{w\in \mathbb{Z}_j^d}\left( |\Phi^0_{U_{j,w}-U_{j,w}}(x)| + \sum_{i = 1}^{r}|\Phi^i_{U_{j,w}-U_{j,w}}(x)|\right)|\Phi_{T-T}(x)| \\
&  \leq    \sum_{x \in \mathbb{Z}^d}|\Phi_{T-T}(x)|  \sum_{j\in[m]}\sum_{w\in \mathbb{Z}_j^d}|\Phi^0_{U_{j,w}-U_{j,w}}(x)| +  \sum_{j\in[m]} \sum_{i = 1}^{r}|V|^{2r-i}\ell^d \\
&  \leq   \sum_{x \in \mathbb{Z}^d}\Phi_{T-T}(x)   +   \sum_{j\in[m]} \sum_{i = 1}^{r-1}|V|^{2r-i}\ell^d\\
& \leq \ell^{2d} + rm|V|^{2r-1}\ell^d,
\end{align*} On the other hand, using \eqref{sqsum} and \eqref{ujwt}, we can compute\begin{align*}
\sum_{j\in [m]}\sum_{w \in \mathbb{Z}_j^d}\frac{|U_{j,w}|^2||T|^2}{|U_{j,w} + T|} & \geq  \sum_{j\in [m]}\sum_{w \in \mathbb{Z}_j^d}\frac{|U_{j,w}|^2\ell^{2d}}{(rn/j + \ell)^d} \\
& \geq     \sum_{j\in [m]}  \frac{|S_r|^2\ell^{2d}}{j^d(rn/j + \ell)^d} \\ 
& =    \sum_{j\in [m]}  \frac{|S_r|^2 \ell^{2d}}{(rn + j\ell)^d}\\
& \geq   \frac{m|S_r|^2\ell^{2d}}{(rn + m\ell)^d},
\end{align*}

Combining the two inequalities above gives us \begin{align*}
    &\frac{m|S_r|^2\ell^{2d}}{(rn + m\ell)^d} \leq \ell^{2d} + rm|V|^{2r-1}\ell^d\\
    \implies& |S_r|^2  \leq \frac{(rn + m\ell)^d}{m} + r|V|^{2r-1}\frac{(rn + m\ell)^d}{\ell^d}.
\end{align*} By setting $m = n^{\frac{d}{2rd + 1}}$ and $\ell = n^{1 -\frac{d}{2rd + 1} }$, we get\begin{equation*}
    \binom{|V|}{r}^2 = |S_r|^2  \leq cn^{d - \frac{d}{2rd + 1}} + c|V|^{2r-1}n^{\frac{d^2}{2rd + 1}},
\end{equation*} for some constant $c$ depending only on $d$ and $r$. We can solve from this inequality that \begin{equation*}
    |V| = O\left(n^{\frac{d}{2r}\left(1 - \frac{1}{2rd + 1}\right)}\right),
\end{equation*} completing the proof.\end{proof}

\section{Concluding remarks}\label{sec_remarks}

1. It is easy to see that $\alpha_{d,s}(N)\geq \Omega(N^{1/d})$ for any fixed $d,s\geq 2$. Let $S$ be a set consisting of $N$ points in $\mathbb{R}^d$ with no $d+s$ members on a hyperplane. Suppose $V$ is a maximal subset of $S$ in general position, then $V$ generates at most $\binom{|V|}{d}$ hyperplanes and each of them covers at most $s$ points from $S\setminus V$. Hence we have the inequality\begin{equation*}
    s\binom{|V|}{d}+|V|\geq |S|=N,
\end{equation*}which justifies the claimed lower bound of $\alpha_{d,s}(N)$.

\begin{problem}
Are there fixed integers $d,s\geq 3$ such that $\alpha_{d,s}(N)\leq o(N^{1/2})$?
\end{problem}

\medskip

\noindent 2. We call a subset $V\subset [n]^d$ a \emph{$m$-fold $B_g$-set} if $V$ only contains trivial solutions to the equations \begin{equation*}
c_1\textbf{x}_1+c_2\textbf{x}_2+\dots+c_g\textbf{x}_g=c_1\textbf{x}'_1+c_2\textbf{x}'_2+\dots+c_g\textbf{x}'_g,
\end{equation*} with constant coefficients $c_i \in [m]$. We call $1$-fold $B_g$-sets simply \emph{$B_g$-sets}. By counting distinct sums, we have an upper bound $|V|\leq O(n^{d/g})$ for any $B_g$-set $V\subset [n]^d$.

Our Theorem~\ref{multifold} can be interpreted as the following phenomenon: by letting $m$ grow as some proper polynomial in $n$, we have an upper bound for $m$-fold $B_g$-sets, where $g$ is even, which gives a polynomial-saving improvement from the trivial $O(n^{d/g})$ bound. We believe this phenomenon should also hold without the parity condition on $g$.

\bibliographystyle{abbrv}
{\footnotesize\bibliography{main}}
\end{document}